\documentclass[11pt]{article}
\usepackage[margin=1.25in]{geometry}

\usepackage{mdwlist}
\usepackage{enumerate}
\usepackage{utf8math}
\usepackage{todonotes}
\usepackage{framed}
\usepackage{xspace}

\usepackage{amsmath,amssymb}
\usepackage{graphics, subfigure, float}
\usepackage{fp, calc}
\usepackage{hyperref}
\usepackage{url}

\usepackage[T1]{fontenc} 
\usepackage{fourier}
\usepackage{bm}

\usepackage{amscd,amsthm}

\usepackage{verbatim, comment}
\usepackage{datetime}

\newtheoremstyle{theorem}{1em}{1em}{\slshape}{0pt}{\bfseries}{.}{ }{}
\theoremstyle{theorem}
\newtheorem{theorem}{Theorem}

\newtheorem*{theorem*}{Theorem}
\newtheorem{corollary}[theorem]{Corollary}

\newtheorem{lemma}[theorem]{Lemma}

\newtheorem*{claim*}{Claim}

\theoremstyle{remark}

\newtheorem*{remark*}{Remark}

\newtheoremstyle{example}{1em}{1em}{}{0pt}{\bfseries}{.}{ }{}

\providecommand{\setN}{\mathbb{N}}
\providecommand{\setZ}{\mathbb{Z}}
\providecommand{\setQ}{\mathbb{Q}}
\providecommand{\setR}{\mathbb{R}}

\newcommand{\Vol}{\textrm{Vol}}

\newcommand{\E}{\mathop{\mathbb{E}}}

\newcommand{\IP}{\textsc{(IP)}\xspace}
\newcommand{\aIP}{\textsc{(Approx-IP)}\xspace}

\usepackage{calrsfs} 
\DeclareMathAlphabet{\pazocal}{OMS}{zplm}{m}{n}

\usepackage[displaymath,textmath]{preview} 
\PreviewEnvironment{center} 

\makeatother

\title{From approximate to exact integer programming\footnote{The conference version of this work was published at IPCO 2023.}}
\date{}
\author{Daniel Dadush\thanks{CWI Email: {\tt dadush@cwi.nl}. Supported by ERC Starting Grant no. 805241-QIP} \and Friedrich Eisenbrand\thanks{EPFL Email {\tt friedrich.eisenbrand@epfl.ch}. Supported by the Swiss National Science Foundation (SNSF) grant 185030 and 207365} \and Thomas Rothvoss  \thanks{University of Washington, Seattle. Email: {\tt rothvoss@uw.edu}. OrcID 0009-0007-8314-0963. Supported by NSF CAREER grant 1651861 and a David \& Lucile Packard Foundation Fellowship.}}

\begin{document}
\maketitle
\begin{abstract}
  \noindent 
Approximate integer programming is the following: For a given  convex body $K ⊆ ℝ^n$, either determine whether $K ∩ ℤ^n$ is empty, or find an integer point in the convex body $2⋅(K - c) +c$ which is $K$, scaled by $2$ from its center of gravity $c$. Approximate integer programming can be solved in time $2^{O(n)}$ while the fastest known methods for exact integer programming run in time $2^{O(n)} ⋅ n^n$. So far, there are no efficient methods for integer programming  known that are based on approximate integer programming. Our  main contribution are two such methods, each yielding novel complexity results.  
 
  First, we show that an  integer point  $x^* ∈(K ∩ℤ^n)$ can be found in time  $2^{O(n)}$, provided that the \emph{remainders} of each component $x_i^* \mod{ℓ}$ for some arbitrarily fixed $ℓ ≥ 5(n+1)$  of $x^*$ are given. The algorithm is based on a \emph{cutting-plane technique}, iteratively halving the volume of the feasible set. The cutting planes are determined via  approximate integer programming. Enumeration of the possible remainders  gives a $2^{O(n)}n^n$ algorithm for general integer programming. This matches the current best  bound of an algorithm by  Dadush (2012) that is considerably more involved. Our algorithm  also relies  on a new \emph{asymmetric approximate Cara\-théodory theorem} that might be of interest on its own. 

  Our second method  concerns integer programming problems in equation-standard form $Ax = b, 0 ≤x ≤ u, \, x ∈ℤ^n$ . Such a problem can be reduced to the solution of $∏_i O(\log u_i +1)$  approximate integer programming problems.  This implies, for example that \emph{knapsack} or  \emph{subset-sum} problems with \emph{polynomial variable range}  $0 ≤ x_i ≤ p(n)$ can be solved in time $(\log n)^{O(n)}$. For these problems, the best running time so far was  $n^n ⋅ 2^{O(n)}$. 
\end{abstract}

\section{Introduction}
\label{sec:introduction}

Many \emph{combinatorial optimization problems} as well as many problems from the \emph{algorithmic geometry of numbers} can be formulated as an \emph{integer linear program}
\begin{equation}
  \label{eq:4}
  \max\{ \left<c,x\right> \mid Ax≤b, x ∈ℤ^n\} 
\end{equation}
where $A ∈ ℤ^{m × n}, b∈ℤ^m$ and $c ∈ℤ^n$, see, e.g.~\cite{GroetschelLovaszSchrijver88,NemhauserWolsey89,Schrijver95}.  Lenstra~\cite{lenstra1983integer}  has shown that integer programming can be solved in polynomial time, if the number of variables is fixed. A careful analysis of his algorithm yields a running time of $2^{O(n^2)}$ times a polynomial in the binary encoding length of the input of the integer program. Kannan~\cite{kannan1987minkowski} has improved this to $n^{O(n)}$, where, from now on we ignore the extra factor that depends polynomially on the input length. At the time, this paper was first submitted, the best algorithm was the one of Dadush~\cite{dadush2012integer} with a running time of $2^{O(n)} ⋅ n^n$.

The question whether there exists a \emph{singly exponential time}, i.e., a  $2^{O(n)}$-algorithm for integer programming is one of the most prominent open problems in the area of algorithms and complexity.
Integer programming can be described in the following more general form. Here, a \emph{convex body}  is synonymous for a full-dimensional compact and convex set.   
  \begin{center}
    \begin{minipage}{1.0\linewidth}
      \begin{framed}
      
        \noindent {\bf Integer Programming (IP)}

        \smallskip

  \noindent 
  Given a \emph{convex body} $K ⊆ ℝ^n$, find an integer  solution $x^* ∈ K ∩ ℤ^n$ or assert that  $K ∩ℤ^n = ∅$.  
\end{framed}
\end{minipage}
\end{center}
  
The convex body $K$ must be well described in the sense that there is access to a  \emph{separation oracle}, see~\cite{GroetschelLovaszSchrijver88}. 
Furthermore, one assumes that $K$ contains a ball of radius $r>0$ and that it is contained in some ball of radius $R$. In this setting,  the current best running times hold as well. The additional polynomial factor in the input encoding length becomes a polynomial factor in $\log (R/r)$ and the dimension $n$.
Central to this paper is \emph{Approximate integer programming} which is as follows. 
 \begin{center}
    \begin{minipage}{1.0\linewidth}
      \begin{framed}
      
        \noindent {\bf Approximate Integer Programming (Approx-IP)}

        \smallskip

  \noindent 

  Given a \emph{convex body} $K ⊆ ℝ^n$, let $c ∈ ℝ^n$ be its center of gravity.  Either find an integer vector  $x^* ∈ (2 ⋅ (K-c) + c)  ∩ ℤ^n$, or assert that  $K ∩ℤ^n = ∅$.  
\end{framed}
\end{minipage}
\end{center} 

The convex body $ 2 ⋅ (K-c) + c$ is $K$ scaled by a factor of $2$ from
its center of gravity. The algorithm of
Dadush~\cite{ApproximateIP-DadushAlgorithmica2014} solves approximate integer
programming in singly exponential time $2^{O(n)}$.  Despite its clear
relation to exact integer programming, there is no reduction from
exact to approximate known so far. Our guiding question is the
following: Can approximate integer programming be used to solve the
exact version of (specific) integer programming programming problems?

\subsection{Contributions of this paper}
\label{sec:contr-this-paper}

We present two different algorithms to reduce the exact integer programming problem~\IP to the approximate version~\aIP.

\begin{enumerate}[a)]  
\item Our first method is a randomized cutting-plane algorithm  that, in time $2^{O(n)}$ and  for any $ℓ ≥ 5(n+1)$  finds a point in $K∩( ℤ^n/ ℓ)$ with high probability, if $K$ contains an integer point.   This algorithm uses an oracle for \aIP on $K$ intersected with one side of a hyperplane that is close to  the center of gravity. Thereby, the algorithm collects  $ℓ$ integer points close to $K$. The collection is such that the convex combination with uniform weights $1/ℓ$  of these points lies in $K$. If, during an iteration, no point is found, the volume of $K$ is roughly halved and eventually $K$ lies on a lower-dimensional subspace on which one can recurse.    \label{item:1}

\item  If equipped with the component-wise remainders  $v ≡ x^* \pmod{ℓ}$ of a solution $x^*$ of \IP, one can use the algorithm to find a point in $(K - v) ∩ ℤ^n $
  and combine it with the remainders to a full solution of \IP,
  using that $(K-v) \cap \ell \setZ^n \neq \emptyset$.
  This  runs in singly exponential randomized time $2^{O(n)}$.   Via enumeration of all remainders,  one obtains an algorithm for \IP that runs in time $2^{O(n)} ⋅ n^n$. This matches the best-known running time for general integer programming~\cite{ApproximateIP-DadushAlgorithmica2014}, which is considerably  involved. 
  \label{item:3} 
  
\item Our analysis depends on a new \emph{approximate
    Cara\-th\'eodory theorem} that we develop in
  Section~\ref{sec:an-asymm-appr}.  While approximate Carath\'eodory
  theorems are known for centrally symmetric convex bodies~\cite{novikoff1963convergence,barman2015approximating,mirrokni2017tight},
  our version is for general convex sets and might be of
  interest on its own.  \label{item:2}

\item Our second method is for integer programming problems $ Ax = b,\, x ∈ℤ^n, \, 0 ≤x ≤ u$ in equation standard form. We show that such a problem can be reduced to $2^{O(n)} \cdot (∏_i \log(u_i+1))$ instances of \aIP. This yields  a running time of $(\log n)^{O(n)}$ for such  IPs, in which the variables are bounded by a polynomial in the dimension. The so-far best running time for such instances was $2^{O(n)} ⋅n^n$ at the time of the first submission of this paper.  Well known benchmark problems in this setting are \emph{knapsack} and \emph{subset-sum} with polynomial upper bounds on the variables, see Section~\ref{sec:ips-with-polynomial}.  \label{item:5} 
\end{enumerate}

\subsection{Related work}
\label{sec:relatd-work}

If the convex body $K$ is an ellipsoid, then the integer programming problem~\IP is the well known \emph{closest vector problem (CVP)} which can be solved in time $2^{O(n)}$ with an algorithm by Micciancio and Voulgaris~\cite{micciancio2010deterministic}. Blömer and Naewe~\cite{DBLP:journals/tcs/BlomerN09} previously  observed that the sampling technique of Ajtai et al.~\cite{ajtai2001sieve} can be modified in such a way as to solve the closest vector approximately.  More precisely, they showed that a $(1+ε)$-approximation of the closest vector problem can be found in time $O(2 + 1/ε)^n$ time. This was later generalized to arbitrary convex sets by Dadush~\cite{ApproximateIP-DadushAlgorithmica2014}. This algorithm either asserts that the convex body $K$ does not contain any integer points, or it finds an integer point in the body stemming from $K$ scaled by $(1+ε)$ from its center of gravity. Also the running time of this randomized algorithm is $O(2 + 1/ε)^n$.  In our paper, we restrict to the case $ε = 1$ which can be solved in singly exponential time. The technique of reflection sets was also used by Eisenbrand et al.~\cite{eisenbrand2011covering} to solve (CVP) in the $ℓ_∞$-norm approximately in time $O(2 + \log (1/ε))^n$.

In the setting in which integer programming can be attacked with dynamic programming, tight upper and lower bounds on the complexity are known~\cite{eisenbrand2019proximity,jansen2018integer,knop2020tight}. Our $n^n ⋅ 2^{O(n)}$ algorithm could be made more efficient by constraining the possible remainders of a solution $\pmod{ℓ}$ efficiently. This barrier is different than the one in  classical integer-programming methods that are based on branching on flat directions~\cite{lenstra1983integer,GroetschelLovaszSchrijver88} as they result in a branching tree of size $n^{O(n)}$.

The \emph{subset-sum problem} is as follows. Given a set $Z ⊆ ℕ$ of $n$ positive integers and a \emph{target value} $t ∈ ℕ$, determine whether there exists a subset $S ⊆ Z$ with $∑_{ s ∈S} s = t$. Subset sum is a classical
 NP-complete problem that serves as a benchmark in algorithm design. The problem can be solved in pseudopolynomial time~\cite{bellman1966dynamic} by dynamic programming. The current fastest pseudopolynomial-time  algorithm is the one of Bringmann\cite{bringmann2017near} that runs in time $O( n+t)$ up to polylogarithmic factors.  There exist instances of subset-sum whose set of feasible solutions, interpreted as $0/1$ incidence vectors, require numbers of value $n^n$ in the input, see~\cite{alon1997anti}. Lagarias and Odlyzko~\cite{lagarias1985solving} have shown that instances of subset sum in which each number of the input $Z$ is drawn uniformly at random from $\{1,\dots,2^{O(n^2)}\}$ can be solved in polynomial time with high probability. The algorithm of Lagarias and Odlyzko is based on the LLL-algorithm~\cite{lenstra1982factoring} for lattice basis reduction.

\subsection{Subsequent work}

After the acceptance of the conference version of this work, Reis and Rothvoss~\cite{ReisRothvossFOCS23} proved that an
algorithm originally suggested by Dadush~\cite{dadush2012integer} can solve any
$n$-variable integer program $\max\{ \left<c,x\right> \mid Ax \leq b, x \in \setZ^n \}$ in time $(\log n)^{O(n)}$
times a polynomial in the encoding length of $A$, $b$ and $c$. However, the question
whether there is a $2^{O(n)}$-time algorithm remains wide open and the approach used by Reis and Rothvoss
inherently cannot provide running times bounds below $(\log n)^{O(n)}$.

\section{Preliminaries}
A \emph{lattice} $\Lambda$ is the set of integer combinations of linearly independent vectors,
i.e. $\Lambda := \Lambda(B) := \{ Bx \mid x \in \setZ^r \}$ where $B \in \setR^{n \times r}$ has linearly independent columns. The \emph{determinant} is the volume of the $r$-dimensional
parallelepiped spanned by the columns of the basis $B$, i.e. $\det(\Lambda) := \sqrt{\det_r(B^TB)}$.
We say that $\Lambda$ has \emph{full rank} if $n=r$. In that case the determinant
is simply $\det(\Lambda)=|\det_n(B)|$.
For a full rank lattice $\Lambda$, we denote the
dual lattice by $\Lambda^* = \{ y \in \setR^n \mid \left<x,y\right> \in \setZ \;\forall x \in \Lambda\}$. Note that $\det(\Lambda^*) \cdot \det(\Lambda) = 1$. For an introduction to lattices, we refer to~\cite{MicciancioGoldwasserLatticeBook2002}. 

A set $Q \subseteq \setR^n$ is called a \emph{convex body} if it is
convex, compact and has a non-empty interior. A set  $Q$ is \emph{symmetric} if $Q = -Q$. Recall that any
symmetric convex body $Q$ naturally induces a norm $\| \cdot \|_Q$ of the form $\|x\|_Q = \min\{ s \geq 0 \mid x \in sQ\}$.
For a full rank lattice $\Lambda \subseteq \setR^n$ and a symmetric convex body $Q \subseteq \setR^n$ we denote
$\lambda_1(\Lambda,Q) := \min\{ \|x\|_Q \mid x \in \Lambda \setminus \{ \bm{0}\} \}$ as the length of the shortest vector
with respect to the norm induced by $Q$. 
We denote the Euclidean ball by $B_2^n := \{ x \in \setR^n \mid \|x\|_2 \leq 1\}$ and the $ℓ_∞$-ball by  $B_{\infty}^n := [-1,1]^n$.
An (origin centered) \emph{ellipsoid} is of the form $\pazocal{E} = A(B_2^n)$ where $A : \setR^n \to \setR^n$ is an invertible linear map.
For any such ellipsoid $\pazocal{E}$ there is a unique positive definite matrix $M \in \setR^{n \times n}$
so that $\|x\|_{\pazocal{E}} = \sqrt{x^TMx}$. 
The \emph{barycenter} (or \emph{centroid}) of a convex body $Q$ is the point $\frac{1}{\Vol_n(Q)}\int_Q x \; dx$.
We will use the following version of \aIP  that runs in time  $2^{O(n)}$, provided that the symmetrizer for the used center $c$ is large enough. This is the case for $c$ being the center of gravity, see Theorem~\ref{thm:MilmanPajorIneq}. Note that the center of gravity of a convex body can be (approximately) computed in randomized polynomial time~\cite{dyer1991randomPolyTimeAlgoForApxVolumeInConvexBody,bertsimas2004solving}.  
\begin{theorem}[Dadush~\cite{ApproximateIP-DadushAlgorithmica2014}] \label{thm:ApxCVP}
There is a $2^{O(n)}$-time algorithm $\textsc{ApxIP}(K,c,\Lambda)$ that takes as input a
convex body $K \subseteq \setR^n$, a point $c \in K$ and a lattice $\Lambda \subseteq \setR^n$. Assuming that 
 $\textrm{Vol}_n( (K-c) \cap (c-K) ) \geq 2^{-\Theta(n)} \textrm{Vol}_n(K)$ the algorithm either returns a point $x \in (c+2(K-c)) \cap \Lambda$ 
or returns \textsc{empty} if $K \cap \Lambda = \emptyset$. 
\end{theorem}
\noindent
One of the classical results in the geometry of numbers is Minkowski's Theorem which we will use in the following form:
\begin{theorem}[Minkowski's Theorem] \label{thm:MinkowskisTheorem}
  For a full rank lattice $\Lambda \subseteq \setR^n$ and a symmetric convex body $Q \subseteq \setR^n$ one has
  \[
   \lambda_1(\Lambda,Q) \leq 2 \cdot \Big(\frac{\det(\Lambda)}{\Vol_n(Q)}\Big)^{1/n}
  \]
\end{theorem}
\noindent 
We will use the following bound on the density of sublattices which is an immediate consequence of
Minkowski's Second Theorem. Here we abbreviate $\lambda_1(\Lambda) := \lambda_1(\Lambda,B_2^n)$.
\begin{lemma} \label{lem:DetOfSublatticeBound}
  Let $\Lambda \subseteq \setR^n$ be a full rank lattice. Then for any $k$-dimensional sublattice $\tilde{\Lambda} \subseteq \Lambda$ one has $\det(\tilde{\Lambda}) \geq (\frac{\lambda_1(\Lambda)}{\sqrt{k}})^k$.
\end{lemma}

Finally, we revisit a few facts from \emph{convex geometry}. Details and proofs can be found in the
excellent textbook by Artstein-Avidan, Giannopoulos and Milman~\cite{AsymptoticGeometricAnalysisBook2015}.
\begin{lemma}[Gr\"unbaum's Lemma] \label{lem:Gr\"unbaumLemma}
  Let $K \subseteq \setR^n$ be any convex body and let $\left<a,x\right> = \beta$ be any hyperplane through the
  barycenter of $K$. Then $\frac{1}{e} \textrm{Vol}_n(K) \leq \textrm{Vol}_n(\{ x \in K \mid \left<a,x\right> \leq \beta\}) \leq (1-\frac{1}{e}) \textrm{Vol}_n(K)$.
\end{lemma}
For a convex body $K$, there are two natural symmetric convex bodies that approximate $K$ in many ways: the
``inner symmetrizer'' $K \cap (-K)$ (provided $\bm{0} \in K$) and the ``outer symmetrizer'' in form of the \emph{difference body} $K-K$. 
The following is a consequence of a more general inequality of Milman and Pajor.
\begin{theorem} \label{thm:MilmanPajorIneq}
  Let $K \subseteq \setR^n$ be any convex body with barycenter $\bm{0}$. Then $\textrm{Vol}_n(K \cap (-K)) \geq 2^{-n}\textrm{Vol}_n(K)$.
\end{theorem}
In particular Theorem~\ref{thm:MilmanPajorIneq} implies that choosing $c$ as the barycenter of $K$ in Theorem~\ref{thm:ApxCVP}
results in a $2^{O(n)}$ running time --- however this will not be the choice that we will later make for $c$. 
Also the size of the difference body can be bounded:
\begin{theorem}[Inequality of Rogers and Shephard] \label{thm:RogersSheppard}
  For any convex body $K \subseteq \setR^n$ one has $\Vol_n(K-K) \leq 4^n \Vol_n(K)$.
\end{theorem}

Recall that for a convex body $Q$ with $\bm{0} \in \textrm{int}(Q)$, the \emph{polar} is $Q^{\circ} = \{ y \in \setR^n \mid \left<x,y\right> \leq 1 \; \forall x \in Q\}$. 
We will use the following relation between volume of a symmetric convex body and the volume of the polar;
to be precise we will use the lower bound (which is due to Bourgain and Milman).
\begin{theorem}[Blaschke-Santal\'o-Bourgain-Milman] \label{thm:BSBM}
  For any symmetric convex body $Q \subseteq \setR^n$ one has
  \[
   C^n \leq \frac{\Vol_n(Q) \cdot \Vol_n(Q^{\circ})}{\Vol_n(B_2^n)^2} \leq 1
  \]
  where $C>0$ is a universal constant.
\end{theorem}

We will also rely on the result of Frank and Tardos to reduce the bit complexity of constraints: 
\begin{theorem}[Frank, Tardos~\cite{frank1987application}\label{thm:FrankTardos}] 
  There is a polynomial time algorithm that takes $(a,b) \in \setQ^{n+1}$ and $\Delta \in \setN_+$ as input
  and produces a pair $(\tilde{a},\tilde{b}) \in \setZ^{n+1}$ with $\|\tilde{a}\|_{\infty},|\tilde{b}| \leq 2^{O(n^3)} \cdot \Delta^{O(n^2)}$
  so that
  $\left<a,x\right> = b \Leftrightarrow \left<\tilde{a},x\right> = \tilde{b}$ and $\left<a,x\right> \leq b \Leftrightarrow \left<\tilde{a},x\right> \leq \tilde{b}$
  for all $x \in \{ -\Delta,\ldots,\Delta\}^n$.
\end{theorem}

\section{The Cut-or-Average algorithm}
\label{sec:cut-or-average}
First, we discuss our $\textsc{Cut-Or-Average}$ algorithm that on input of a convex set $K$,
a lattice $\Lambda$ and integer $\ell \geq 5(n+1)$, either finds a point $x \in \frac{\Lambda}{\ell} \cap K$ or decides
that $K \cap \Lambda = \emptyset$ in time $2^{O(n)}$.
Note that for any polyhedron $K = \{ x \in \setR^n \mid Ax \leq b\}$ with rational $A,b$ and lattice $\Lambda$ with basis $B$ one can compute a value of $\Delta$ so that $\log(\Delta)$ is polynomial in the
encoding length of $A$, $b$ and $B$ and $K \cap \Lambda \neq \emptyset$ if and only if
$K \cap [-\Delta,\Delta]^n \cap \Lambda \neq \emptyset$. See Schrijver~\cite{TheoryOfLPandIP-Schrijver1999} for details. In other words, w.l.o.g.
we may assume that our convex set is bounded.
The pseudo code of the algorithm can be found in Figure~\ref{fig:CutOrAveAlgo}. 
\begin{figure} 
\begin{center}
  \begin{minipage}{15cm}
    \hrule \vspace{2mm}
{\bf Input:} Convex set $K \subseteq \setR^{n}$, lattice $\Lambda$, parameter $\ell \geq 5(n+1)$ \\
{\bf Output:} Either a point $x \in K \cap \frac{\Lambda}{\ell}$ or conclusion that $K \cap \Lambda = \emptyset$
\begin{enumerate*}
\item[(1)] WHILE $\lambda_1(\Lambda^*,(K-K)^{\circ}) > \frac{1}{2}$ DO
  \begin{enumerate*}
\item[(2)] Compute barycenter $c$ of $K$.
\item[(3)] Let $\pazocal{E} := \{x \in \setR^n \mid x^T M x \leq 1\}$ be a $\bm{0}$-centered ellipsoid with $c+\pazocal{E} \subseteq K \subseteq c+R\cdot \pazocal{E}$ for $R := n+1$, let $\rho := \frac{1}{4n}$.
\item[(4)] Let $z := \textsc{ApxIP}(K, c, \Lambda)$, $X = \{z\}$. If $z = \textsc{empty}$, Return $\textsc{empty}$.
\item[(5)] WHILE $\|c-z\|_\pazocal{E} > \frac{1}{4}$ DO
  \begin{enumerate*}
  \item[(6)] Let $d := \frac{-(z-c)}{\|z-c\|_{\pazocal{E}}}$, $a := -M(z-c)$.  
  \item[(7)] Compute $x := \textsc{ApxIP}(K \cap \{ x \in \setR^n \mid \left<a,x\right> \geq \left<a,c+ \rho d/2\right> \}, c+\rho d, \Lambda)$.
  \item[(8)] IF $x = \textsc{empty}$ THEN replace $K$ by $K' := K \cap \{ x \in \setR^n \mid \left<a,x\right> \leq \left<a,c+\rho d/2\right> \}$ and GOTO (1).
  \item[(9)] ELSE $X := X \cup \{x\}$, $z := (1-\frac{1}{|X|})z + \frac{x}{|X|}$.
  \end{enumerate*}
\item[(10)] Compute $\mu \in \frac{\setZ_{\geq 0}^X}{\ell}$ with $\sum_{x \in X} \mu_x=1$ and $\sum_{x \in X} \mu_x x \in K$  using Asymmetric Approximate Carath\'eodory (see Sec~\ref{sec:an-asymm-appr}).
  \item[(11)] Return $\sum_{x \in X} \mu_xx$.
\end{enumerate*}
\item[(12)] Compute $y \in \Lambda^* \setminus \{ \bm{0}\}$ with $\|y\|_{(K-K)^{\circ}} \leq \frac{1}{2}$.
\item[(13)] Find $\beta \in \setZ$ so that $K \cap \Lambda \subseteq U$ with $U = \{ x \in \setR^n  \mid \left<y,x\right> = \beta\}$.
\item[(14)] IF $n=1$ THEN $U = \{ x^*\}$; Return $x^*$ if $x^* \in \Lambda$ and return ``$K \cap \Lambda = \emptyset$'' otherwise. 
\item[(15)] Recurse on $(n-1)$-dim. instance $\textsc{Cut-Or-Average}(K \cap U, \Lambda \cap U, \ell)$. 
\end{enumerate*} \hrule
\end{minipage}
\caption{The Cut-Or-Average algorithm.\label{fig:CutOrAveAlgo}}
 \end{center}
\end{figure}
 An intuitive description of the algorithm is as follows: we compute the barycenter $c$ of $K$ and an ellipsoid
 $\pazocal{E}$ that approximates $K$ up to a factor of $R = n+1$.
  The goal is to find a point $z$ close to the barycenter $c$ so that $z$ is a convex combination
 of lattice points that all lie in a 3-scaling of $K$. We begin by choosing $z$ as any such lattice
 vector and then iteratively update $z$ using the oracle for approximate integer programming from Theorem~\ref{thm:ApxCVP} to move closer to $c$.
 If this succeeds, then we can directly use an asymmetric version
 of the \emph{Approximate Carath\'eodory Theorem} (Lemma~\ref{lem:AsymApxCaratheodory}) to find an unweighted average of
 $\ell$ lattice points that lies in $K$; this would be a vector of the form $x \in \frac{\Lambda}{\ell} \cap K$. If the algorithm fails to approximately express $c$ as a convex combination of lattice points,
 then we will have found a hyperplane $H$ going almost through the barycenter $c$ so that $K \cap H_{\geq}$ does not contain a lattice point. Then the algorithm continues searching in $K \cap H_{\leq}$. This case might happen repeatedly, but after polynomial number of times, the volume of $K$ will have dropped below
 a threshold so that we may recurse on a \emph{single} $(n-1)$-dimensional subproblem. 
 We will now give the detailed analysis. Note that in order to obtain a clean exposition we did not aim to optimize any
 constant. However by merely tweaking the parameters one could make the choice of $\ell = (1+\varepsilon)n$ work
 for any constant $\varepsilon > 0$.
\begin{figure}
 \begin{center}
  \includegraphics{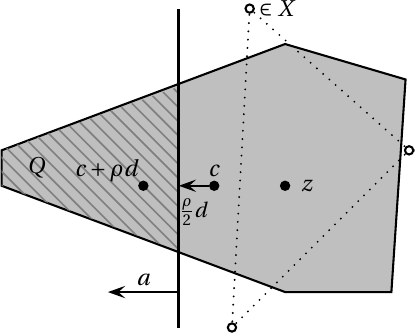}
\end{center}
  \caption{Visualization of the inner WHILE loop where $Q := K \cap \{ x \in \setR^n \mid \left<a,x\right> \geq \left<a,c+\frac{\rho}{2}d\right>\}$.}
\end{figure}

\subsection{Bounding the number of iterations}
We begin the analysis with a few estimates that will help us to bound the number of iterations.
\begin{lemma} \label{lem:XisIn3ScalingOfK}
  Any point $x$ found in line (7) lies in a 3-scaling of $K$ around $c$, i.e. $x \in c + 3(K-c)$ assuming $0 <\rho \leq 1$.
\end{lemma}
\begin{proof}
  We verify that
\[
x \in (c-\rho d) + 2(K-(c-\rho d)) = c + 2(K-c) + \rho d \subseteq c + 3(K-c) 
\]
 using that $\|\rho d\|_{\pazocal{E}} = \rho \leq 1$.
\end{proof}
Next we bound the distance of $z$ to the barycenter: 
\begin{lemma}\label{lem:distance}
  At the beginning of the $k$th iterations of the WHILE loop on line (5), one has $\|c-z\|_{\pazocal{E}}^2 \leq \frac{9 R^2}{k}$. 
\end{lemma}
\begin{proof}
We prove the statement by induction on $k$. At $k=1$, by construction on
line (4), $z \in c+2(K-c) \subseteq c + 2R \pazocal{E}$. Thus
$\|c-z\|_{\pazocal{E}}^2 \leq (2R)^2 \leq 9R^2$, as needed.

Now assume $k \geq 2$. Let $z,z'$ denote the values of $z$ during iteration $k-1$
before and after the execution of line (9) respectively, and let $x$ be the
vector found on line (7) during iteration $k-1$. Note that $z' = (1-\frac{1}{k}) z +
\frac{1}{k} x$. By the induction hypothesis, we have that $\|z-c\|_{\pazocal{E}}^2 \leq
9 R^2/(k-1)$. Our goal is to show that $\|z'-c\|_{\pazocal{E}}^2 \leq 9 R^2/k$.
In (6), we define $d$ as the normalized version of $z-c$ with
$\|d\|_{\pazocal{E}} =  1$ and hence $d \in K-c$. By construction  $\langle a, x-c \rangle \geq 0$
and from Lemma~\ref{lem:XisIn3ScalingOfK} we have $x \in c + 3(K-c)$ which implies $\|x-c\|_{\pazocal{E}} \leq 3R$.
The desired bound on the $\pazocal{E}$-norm of $z'-c$ follows from the following calculation:
\begin{align*}
\|z'-c\|_{\pazocal{E}}^2 &= \Big\|\Big(1-\frac{1}{k}\Big)(z-c) + \frac{1}{k}(x-c)\Big\|_{\pazocal{E}}^2
\\
&= \Big(1-\frac{1}{k}\Big)^2 \|z-c\|_{\pazocal{E}}^2 -2\Big(1-\frac{1}{k}\Big) \frac{1}{k} \langle a, x-c \rangle + \frac{1}{k^2}\|x-c\|_{\pazocal{E}}^2 \\
&\leq \Big(1-\frac{1}{k}\Big)^2 \|z-c\|_{\pazocal{E}}^2 + \frac{1}{k^2}\|x-c\|_{\pazocal{E}}^2 \\
&\leq \Big(\Big(1-\frac{1}{k}\Big)^2 \frac{1}{k-1} + \frac{1}{k^2}\Big) \cdot 9R^2 = \frac{9R^2}{k}.
\end{align*}

\end{proof}
In particular Lemma~\ref{lem:distance} implies an upper bound on the number of iterations of the inner WHILE loop: 
\begin{corollary}
The WHILE loop on line (5) never takes more than $36R^2$ iterations.
\end{corollary}
\begin{proof}
By Lemma~\ref{lem:distance}, for $k := 36R^2$ one has $\|c-z\|_{\pazocal{E}}^2 \leq \frac{9R^2}{k} \leq \frac{1}{4}$.
\end{proof}
Next, we prove that every time we replace $K$ by $K' \subset K$ in line (8), its volume drops by a constant
factor.  
\begin{lemma} \label{lem:VolumeDecrease}
  In step (8) one has $\textrm{Vol}_n(K') \leq (1-\frac{1}{e}) \cdot (1+\frac{\rho}{2})^n \cdot \textrm{Vol}_n(K)$ for any $\rho \geq 0$.
  In particular for $0 \leq \rho \leq \frac{1}{4n}$ one has $\textrm{Vol}_n(K') \leq \frac{3}{4} \textrm{Vol}_n(K)$.
\end{lemma}
\begin{proof}
  The claim is invariant under affine linear transformations, hence we may assume w.l.o.g. that $\pazocal{E} = B_2^n$, $M = I_n$ and $c=\bm{0}$.
 Note that then  $B_2^n \subseteq K \subseteq R B_2^n$. 
Let us abbreviate $K_{\leq t} := \{ x \in K \mid \left<d,x\right> \leq t\}$. 
In this notation $K' = K_{\leq \rho/2}$. Recall that Gr\"unbaum's Lemma (Lemma~\ref{lem:Gr\"unbaumLemma}) guarantees that $\frac{1}{e} \leq \frac{\textrm{Vol}_{n}(K_{\leq 0})}{\textrm{Vol}_n(K)} \leq 1-\frac{1}{e}$. Moreover, it is well known 
that the function $t \mapsto \textrm{Vol}_{n}(K_{\leq t})^{1/n}$ is concave on its support, see again \cite{AsymptoticGeometricAnalysisBook2015}.
Then
\[
  \textrm{Vol}_n(K_{\leq 0})^{1/n} \geq \Big(\frac{1}{1+\rho/2}\Big) \cdot \textrm{Vol}_n(K_{\leq \rho/2})^{1/n} + \Big(\frac{\rho/2}{1+\rho/2}\Big) \cdot \underbrace{\textrm{Vol}_n(K_{\leq -1})^{1/n}}_{\geq 0} \geq \Big(\frac{1}{1+\rho/2}\Big) \cdot \textrm{Vol}_n(K_{\leq \rho/2})^{1/n}
\]
and so
\[
\Big(1-\frac{1}{e}\Big) \cdot \textrm{Vol}_n(K) \geq \textrm{Vol}_n(K_{\leq 0}) \geq \Big(\frac{1}{1+\rho/2}\Big)^n \cdot \textrm{Vol}_n(K_{\leq \rho/2}) 
\]
Rearranging gives the first claim in the form $\textrm{Vol}_n(K_{\leq \rho/2}) \leq (1-\frac{1}{e}) \cdot (1+\frac{\rho}{2})^n \cdot \textrm{Vol}_n(K)$. For the 2nd part we verify that for $\rho \leq \frac{1}{4n}$ one has $(1-\frac{1}{e}) \cdot (1+\frac{\rho}{2})^n \leq (1-\frac{1}{e}) \cdot \exp(\frac{\rho}{2}) \leq \frac{3}{4}$.
\end{proof}

\begin{lemma} \label{lem:IterationNumber}
  Consider a call of $\textsc{Cut-Or-Average}$ on $(K,\Lambda)$ where $K \subseteq r B_2^n$ for some $r>0$. Then the total
  number of iterations of the outer WHILE loop over all recursion levels is bounded by $O(n^2 \log(\frac{nr}{\lambda_1(\Lambda)}))$.
\end{lemma}
\begin{proof}
  Consider any recursive run of the algorithm. The convex set will be of the form $\tilde{K} := K \cap U$
  and the lattice will be of the form $\tilde{\Lambda} := \Lambda \cap U$ where $U$ is a subspace
  and we denote $\tilde{n} := \dim(U)$. We think of $\tilde{K}$ and $\tilde{\Lambda}$ as $\tilde{n}$-dimensional
  objects. 
  Let $\tilde{K}_t \subseteq \tilde{K}$
  be the convex body after $t$ iterations of the outer WHILE loop.
  Recall that $\Vol_{\tilde{n}}(\tilde{K}_t) \leq (\frac{3}{4})^t \cdot \Vol_{\tilde{n}}(\tilde{K})$ by Lemma~\ref{lem:VolumeDecrease}
  and $\Vol_{\tilde{n}}(\tilde{K}) \leq r^{\tilde{n}} \Vol_{\tilde{n}}(B_2^{\tilde{n}})$. Our goal is to show that for $t$ large enough, there is a non-zero lattice vector $y \in \tilde{\Lambda}^*$ with $\|y\|_{(\tilde{K}_t-\tilde{K}_t)^{\circ}} \leq \frac{1}{2}$ which then causes the algorithm to recurse, see Figure~\ref{fig:Recursion}. 
  To prove existence of such a vector $y$, we use Minkowski's Theorem (Theorem~\ref{thm:MinkowskisTheorem}) followed by the Blaschke-Santal\'o-Bourgain-Milman Theorem (Theorem~\ref{thm:BSBM}) to obtain
  \begin{eqnarray*}
    \lambda_1( \tilde{\Lambda}^*, (\tilde{K}_t-\tilde{K}_t)^{\circ}) &\stackrel{\textrm{Thm~\ref{thm:MinkowskisTheorem}}}{\leq}& 2 \cdot \Big(\frac{\det(\tilde{\Lambda}^*)}{\Vol_{\tilde{n}}( (\tilde{K}_t-\tilde{K}_t)^{\circ})}\Big)^{1/\tilde{n}} \\ &\stackrel{\textrm{Thm~\ref{thm:BSBM}}}{\leq}& 2C \cdot \Big(\frac{\Vol_{\tilde{n}}(\tilde{K}_t-\tilde{K}_t)}{\det(\tilde{\Lambda}) \cdot \Vol_{\tilde{n}}(B_2^{\tilde{n}})^2} \Big)^{1/\tilde{n}} \\
                                                                     &\stackrel{\textrm{Thm~\ref{thm:RogersSheppard}}}{\leq}& 2 \cdot 4 \cdot \frac{\sqrt{\tilde{n}}}{2} \cdot C \Big(\frac{\Vol_{\tilde{n}}(\tilde{K}_t)}{\det(\tilde{\Lambda}) \cdot \Vol_{\tilde{n}}(B_2^{\tilde{n}})} \Big)^{1/\tilde{n}} \\
                                                                     &\leq& 4C\sqrt{\tilde{n}} \cdot r \cdot \frac{(3/4)^{t/\tilde{n}}}{\det(\tilde{\Lambda})^{1/\tilde{n}}} \\
    &\leq& 4 C \cdot \frac{\tilde{n} \cdot r}{\lambda_1(\Lambda)} \cdot (3/4)^{t/\tilde{n}}
  \end{eqnarray*}
  Here we use the convenient estimate of $\Vol_{\tilde{n}}(B_2^{\tilde{n}}) \geq \Vol_{\tilde{n}}(\frac{1}{\sqrt{\tilde{n}}}B_{\infty}^{\tilde{n}}) = (\frac{2}{\sqrt{\tilde{n}}})^{\tilde{n}}$.
  Moreover, we have used that by Lemma~\ref{lem:DetOfSublatticeBound} one has $\det(\tilde{\Lambda}) \geq (\frac{\lambda_1(\Lambda)}{\sqrt{\tilde{n}}})^{\tilde{n}}$. 
  Then $t = \Theta(\tilde{n} \log(\frac{\tilde{n}r}{\lambda_1(\Lambda)}))$ iterations suffice until $\lambda_1( \tilde{\Lambda}^*, (\tilde{K}_t-\tilde{K}_t)^{\circ}) \leq \frac{1}{2}$ and the algorithm recurses. 
  Hence the total number of iterations of the outer WHILE loop over all recursion levels can be bounded by $O(n^2 \log(\frac{nr}{\lambda_1(\Lambda)}))$.
\end{proof}
\begin{figure}
 \begin{center}
  \includegraphics{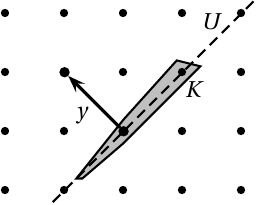}
\end{center}
  \caption{Visualization of lines (12)+(13) (with $n=2$ and $\Lambda=\setZ^2 = \Lambda^*$)\label{fig:Recursion}}
\end{figure}
The iteration bound of Lemma~\ref{lem:IterationNumber} can be improved by amortizing the volume reduction over the different recursion levels following the approach of Jiang~\cite{MinConvexFctWithIntegralMinimizers-Jiang-SODA21}. We refrain from that to keep our approach simple.

\subsection{Running times of the subroutines}

We have already seen that the number of iterations of the {\sc Cut-or-Average} algorithm is polynomially bounded.
Goal of this subsection is to prove that all used subroutines can be implemented in time that
is single-exponential or less. First we prove that steps (2)+(3) take polynomial time.
\begin{lemma}
  For any convex body $K \subseteq \setR^n$ one can compute the barycenter $c$  and
  a $\bm{0}$-centered ellipsoid $\pazocal{E}$ in randomized polynomial time so that $c + \pazocal{E} \subseteq K \subseteq c + (n+1) \pazocal{E}$.
\end{lemma}
\begin{proof}
  We say that a convex body $Q \subseteq \setR^n$ is  \emph{centered and isotropic} if the uniform random sample $X \sim Q$
  satisfies the following conditions: (i) $\E[X]=\bm{0}$ and (ii) $\E[XX^T] = I_n$. For any convex body $K$ one can compute
  an affine linear transformation $T : \setR^n \to \setR^n$ in polynomial time\footnote{At least up to negligible error terms.} so that $T(K)$ is centered and isotropic; this can be done for example by obtaining polynomially many samples of $X$, see \cite{DBLP:journals/rsa/KannanLS97, EstimatesForLogConcaveSamples2010}. A result by Kannan, Lov\'asz and Simonovits (Lemma 5.1 in \cite{DBLP:journals/rsa/KannanLS97}) then says that any such centered and isotropic body $T(K)$ satisfies $B_2^n \subseteq T(K) \subseteq (n+1) B_2^n$. Then $c := T^{-1}(\bm{0})$ and $\pazocal{E} := T^{-1}(B_2^n)-c$ satisfy the claim. 
\end{proof}

In order for the call of $\textsc{ApxIP}$ in step (7) to be efficient, we need that the symmetrizer of the
set is large enough volume-wise, see Theorem~\ref{thm:ApxCVP}. We will prove now that this is indeed the case. In particular for any parameters $2^{-\Theta(n)} \leq \rho \leq 0.99$ and $R \leq 2^{O(n)}$
we will have  $\textrm{Vol}_n((Q-\tilde{c}) \cap (\tilde{c}-Q)) \geq 2^{-\Theta(n)} \textrm{Vol}_n(Q)$ which suffices for our purpose.

\begin{lemma} \label{lem:VolumeForSymmetrizerOfQ}
  In step (7), the set $Q := \{ x \in K \mid \left<a,x\right> \geq \left<a,c+\frac{\rho}{2} d\right>\}$ and the point $\tilde{c} := c+\rho d$
  satisfy $\textrm{Vol}_n((Q-\tilde{c}) \cap (\tilde{c}-Q)) \geq (1-\rho)^n \cdot \frac{\rho}{2R} \cdot 2^{-n}  \cdot \textrm{Vol}_n(Q)$.
\end{lemma}
\begin{proof} 
  Consider the symmetrizer $K' := (K-c) \cap (c-K)$ which has $\textrm{Vol}_n(K') \geq 2^{-n}\textrm{Vol}_n(K)$ by Theorem~\ref{thm:MilmanPajorIneq}
  as $c$ is the barycenter of $K$.
  Set $K'' := \{ x \in K' \mid |\left<a,x\right>| \leq \frac{\rho}{2}|\left<a,d\right>|\}$. As $K'$ is symmetric and all $x \in K'$ satisfy $|\left<a,x\right>| \leq R|\left<a,d\right>|$, we have $
    \textrm{Vol}_n(K'') \geq \frac{\rho}{2R} \textrm{Vol}_n(K') 
    $.
    Now consider
\begin{eqnarray*}
 P &:=& (1-\rho)(K''+c) + \rho (c+d) \\ &=& (1-\rho) K'' + (c+\rho d) \\
&\stackrel{(*)}{\subseteq}& K \cap \big\{ x \in \setR^n : \left<a,x\right> \geq \left<a,c + \frac{\rho}{2}d\right>\big\} = Q.
\end{eqnarray*}
\begin{figure}
  \begin{center}
   \includegraphics{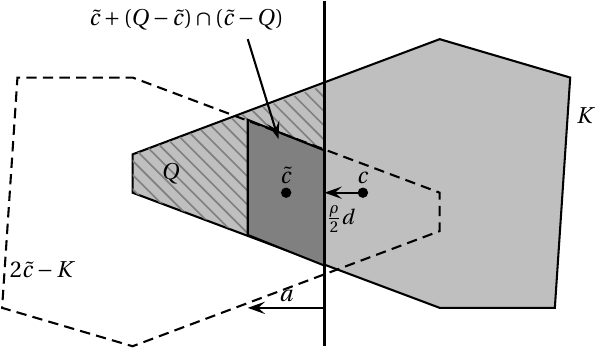}
 \end{center}
   \caption{Visualization of the proof of Lemma~\ref{lem:VolumeForSymmetrizerOfQ} where $\tilde{c} = c + \rho d$.}
 \end{figure}
  For the inclusion in $(*)$ we use that $K'' + c \subseteq K$ and $c+d \in K$;
moreover for any $x \in K''$ one has $\left<a,(1-\rho)x+c+\rho d\right> \geq \left<a,c+\frac{\rho}{2}d\right>$. Finally, $P$ is symmetric about $c+\rho d$ and hence
\[
 \textrm{Vol}_n((Q-\tilde{c}) \cap (\tilde{c}-Q)) \geq \textrm{Vol}_n(P) \geq (1-\rho)^n \cdot \frac{\rho}{2R} \cdot 2^{-n} \cdot \textrm{Vol}_n(Q)
\]
as $Q \subseteq K$.
\end{proof}
Step (10) can be done in polynomial time and we defer the analysis to Section~\ref{sec:an-asymm-appr}.
Step (12) corresponds to finding a shortest non-zero vector in a lattice w.r.t. norm $\| \cdot \|_{(K-K)^{\circ}}$ which can be done in time $2^{O(n)}$ using the Sieving algorithm~\cite{DBLP:conf/stoc/AjtaiKS01}.

\subsection{Conclusion on the Cut-Or-Average algorithm}

From the discussion above, we can summarize the performance of the algorithm in Figure~\ref{fig:CutOrAveAlgo} as follows: 
\begin{theorem} \label{thm:FullAlgorithmStatementI}
  Given a full rank matrix $B \in \setQ^{n \times n}$ and parameters $r>0$ and $\ell \geq 5(n+1)$ with $\ell \in \setN$ and a separation oracle for a closed convex set $K \subseteq r B_2^n$, there is a randomized
  algorithm that with high probability finds a point $x \in K \cap \frac{1}{\ell}\Lambda(B)$ or decides that $K \cap \Lambda(B) = \emptyset$. Here the running time is $2^{O(n)}$ times a polynomial in $\log(r)$ and the encoding length of $B$. 
\end{theorem}
This can be easily turned into an algorithm to solve integer linear programming:
\begin{theorem} \label{thm:FullAlgorithmStatementII}
  Given a full rank matrix $B \in \setQ^{n \times n}$, a parameter $r>0$ and a separation oracle
  for a closed convex set $K \subseteq r B_2^n$, there is a randomized
  algorithm that with high probability finds a point $x \in K \cap \Lambda(B)$ or decides that
  there is none. The running time is $2^{O(n)}n^n$ times a polynomial in $\log(r)$
  and the encoding length of $B$.
\end{theorem}
\begin{proof}
  Suppose that $K \cap \Lambda \neq \emptyset$ and fix an (unknown) solution
  $x^* \in K \cap \Lambda$.
  We set $\ell := \lceil 5(n+1) \rceil$. We iterate through all $v \in \{ 0,\ldots,\ell-1\}^n$ and run Theorem~\ref{thm:FullAlgorithmStatementI} on the set $K$ and the shifted lattice $v+\ell \Lambda$. For the outcome of $v$ with $x^* \equiv v \mod \ell$
  one has $K \cap (v+\ell \Lambda) \neq \emptyset$ and so the algorithm will discover
  a point $x \in K \cap (v+\Lambda)$.
\end{proof}

\section{An asymmetric Approximate Carath\'eodory Theorem}
\label{sec:an-asymm-appr}

In this section we show correctness of (10) and prove that given lattice points $X \subseteq \Lambda$
that are contained in a in a 3-scaling of $K$ and satisfy $c \in \textrm{conv}(X)$, we can find a point in $\frac{\Lambda}{\ell} \cap K$.
The \emph{Approximate Carath\'eodory Theorem} states the following.
\begin{quote}
  Given any point-set $X \subseteq B_2^n$ in the unit ball with $\bm{0} \in \textrm{conv}(X)$ and a parameter $k \in \setN$,
  there exist $u_1,\ldots,u_k \in X$ (possibly with repetition) such that
  \begin{displaymath}
    \left\|\frac{1}{k}\sum_{i=1}^k u_i\right \|_2 \leq O\left(1/\sqrt{k}\right).
  \end{displaymath}
\end{quote}
The theorem is proved, for example, by Novikoff~\cite{novikoff1963convergence} in the context of the \emph{perceptron algorithm}.  An  $ℓ_p$-version  was provided by Barman~\cite{barman2015approximating} to find Nash equilibria. Deterministic and nearly-linear time methods to find the convex combination were recently described in~\cite{mirrokni2017tight}. 
In the following, we provide  a generalization to asymmetric convex bodies and  the dependence on $k$ will be weaker but sufficient for our analysis of our {\sc Cut-or-Average} algorithm from Section~\ref{sec:cut-or-average}.

Recall that with a symmetric convex body $K$, we one can associate
the \emph{Minkowski norm} ${\| \cdot \|_K}$ with $\|x\|_K = \inf\{ s \geq 0 \mid x \in sK\}$.
In the following we will use the same definition also for an arbitrary convex set $K$
with $\bm{0} \in K$. Symmetry is not given but one still has $\|x+y\|_K \leq \|x\|_K + \|y\|_K$ for all
$x,y \in \setR^n$ and $\|\alpha x\|_K = \alpha\|x\|_K$ for $\alpha \in \setR_{\geq 0}$.
Using this notation we can prove the main result of this section.

\begin{lemma} \label{lem:AsymApxCaratheodory}

  Given a point-set $X \subseteq K$ contained in a convex set $K ⊆ ℝ^n$ with $\bm{0} \in \textrm{conv}(X)$ and a parameter $k \in \setN$,
  there exist $u_1,\ldots,u_k \in X$ (possibly with repetition) so that
  \begin{displaymath}
    \left\|\frac{1}{k}\sum_{i=1}^k u_i\right\|_K  \leq {\min\{|X|,n+1\}}/ {k}. 
  \end{displaymath}
  Moreover, given $X$ as input, the points $u_1,\ldots,u_k$ can be found in time polynomial in $|X|$, $k$ and $n$.
\end{lemma}
  \begin{proof}
    Let $ℓ = \min\{|X|, n+1\}$.  The claim is true whenever $k \leq ℓ$
    since then we may simply pick an arbitrary point in $X$.  Hence
    from now on we assume $k>ℓ$.

    By Carathéodory's theorem, there
    exists a convex combination of zero, using $ℓ$ elements of $X$.
    We write $\bm{0} = \sum_{i=1}^ℓ \lambda_i v_i$ where $v_i ∈X$,
    $\lambda_i \geq 0$ for $i \in [ℓ]$ and
    $\sum_{i=1}^ℓ \lambda_i = 1$.
    Consider the numbers $L_i = (k-ℓ) λ_i +1$. Clearly, $∑_{i=1}^ℓ L_i = k.$ This implies that there exists an integer vector $μ ∈ ℕ^ℓ$ with $μ ≥ (k-ℓ) λ$ and $∑_{i=1}^ℓμ_i = k$.
   It remains to show that we have
   \begin{displaymath}
    \left\|\frac{1}{k}\sum_{i=1}^ℓ μ_i v_i \right\|_K  \leq ℓ / {k}.  
  \end{displaymath}
    In fact, one has
   \begin{eqnarray*}
     \Big\|\sum_{i=1}^ℓ \mu_iv_i\Big\|_K &=& \Big\|\sum_{i=1}^ℓ \underbrace{(\mu_i-(k-ℓ)\lambda_i)}_{\geq 0}v_i +\underbrace{(k-ℓ)}_{\geq 0} \sum_{i=1}^ℓ \lambda_iv_i\Big\|_K \\
                                         &\leq& \sum_{i=1}^ℓ (\mu_i-(k-ℓ)\lambda_i) \underbrace{\|v_i\|_K}_{\leq 1} + (k-ℓ) \underbrace{\Big\|\sum_{i=1}^ℓ \lambda_iv_i\Big\|_K}_{=0}\\
     & \leq & ℓ. 
   \end{eqnarray*}
   For the moreover part, note that the coefficients $\lambda_1,\ldots,\lambda_{\ell}$
   are the extreme points of a linear program which can be found in polynomial time.
   Finally, the linear system $\mu \geq \lceil (k-\ell) \lambda \rceil, \sum_{i=1}^{\ell} \mu_i = k$ has a totally unimodular constraint matrix and the right hand side is integral, hence any extreme point solution is integral as well, see e.g.~\cite{TheoryOfLPandIP-Schrijver1999}.
\end{proof}

\begin{lemma}
  For any integer $\ell \geq 5(n+1)$, the convex combination $\mu$ computed in line (10)
  satisfies $\sum_{x \in X} \mu_x x \in K$.
\end{lemma}
\begin{proof}
  We may translate the sets $X$ and $K$ so that $c=\bm{0}$ without affecting the claim. Recall that $z \in \textrm{conv}(X)$.
  By Carath\'eodory's Theorem there are $v_1,\ldots,v_m \in X$ with $m \leq n+1$ so that $z \in \textrm{conv}\{v_1,\ldots,v_m\}$
  and so $\bm{0} \in \textrm{conv}\{ v_1-z,\ldots,v_m-z\}$. We have $v_i \in 3K$ by Lemma~\ref{lem:XisIn3ScalingOfK} and $-z \in \frac{1}{4}\pazocal{E} \subseteq \frac{1}{4}K$ as well as $z \in \frac{1}{4}K$.
  Hence $\|v_i-z\|_K \leq \|v_i\|_K + \|-z\|_K \leq \frac{13}{4}$. We apply Lemma~\ref{lem:AsymApxCaratheodory}
  and obtain a convex combination $\mu \in \frac{\setZ_{\geq 0}^m}{\ell}$ with $\|\sum_{i=1}^m \mu_i(v_i-z)\|_{\frac{13}{4}K} \leq \frac{m}{\ell}$.
    Then
    \[
      \Big\|\sum_{i=1}^m \mu_i v_i\Big\|_K \leq \Big\|\sum_{i=1}^m \mu_i(v_i-z)\Big\|_K + \underbrace{\|z\|_K}_{\leq 1/4} \leq  \frac{13}{4} \frac{m}{\ell} + \frac{1}{4} \leq 1
    \]
    if $\ell \geq \frac{13}{3}m$. This is satisfies if $\ell \geq 5(n+1)$.
\end{proof}

\section{IPs with polynomial variable range}
\label{sec:ips-with-polynomial}

Now we come to our second method that reduces \IP to \aIP that applies
to integer programming in \emph{standard equation form}
\begin{equation}
  \label{eq:1}
   Ax =  b, \,x \in \setZ^n, \, 0 ≤ x_i≤ u_i, \, i=1,\dots,n , 
\end{equation}
Here, $A ∈ ℤ^{m ×n}$, $b ∈ℤ^m$, and the $u_i∈ ℕ_{+}$ are positive integers
that bound the variables from above. Our main goal is to prove the
following theorem.

\begin{theorem}
  \label{thr:1}
  The integer feasibility problem~\eqref{eq:1} can be solved in time $2^{O(n)} \prod_{i=1}^n \log_2 (u_i+1)$. 
\end{theorem}

We now describe the algorithm. It is again based on the approximate integer programming technique of Dadush~\cite{ApproximateIP-DadushAlgorithmica2014}. We exploit it to solve integer programming exactly  via the technique of \emph{reflection sets} developed by Cook et al.~\cite{cook1992integer}.  For each $i = 1,\dots,n$ we consider the two families of hyperplanes  that slice the feasible region with the shifted lower and upper bounds respectively 
\begin{equation}
  \label{eq:2} 
  x_i = 2^{j-1} \text{ and } x_i = u_i -  2^{j-1}, \, 0 ≤ j ≤ ⌈\log_2(u_i)⌉. 
\end{equation}
Following~\cite{cook1992integer}, we consider two points $w,v$ that
lie in the region between two consecutive planes  
$x_i = 2^{j-1}$ and $x_i = 2^{j}$ for some $j$, see Figure~\ref{fig:1}.
\begin{figure}[h]
  \centering
  \includegraphics{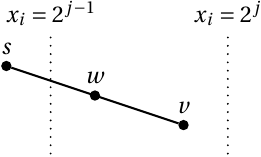} 
  \caption{The reflection set.}
  \label{fig:1}
\end{figure}
Suppose that $w_i ≤v_i$ holds. Let $s$ be the point such that $w = 1/2(s + v)$. The line-segment $s,v$ is
the line segment $w,v$ scaled by a factor of $2$ from $v$. Let us consider what can be said about the $i$-th component of $s$.
Clearly $s_i ≥ 2^{j-1} - (2^j -2^{j-1}) = 0$. Similarly, if $w$ and $v$ lie in the region in-between $x_i = 0$
and $x_i = 1/2$, then $s_i ≥ -1/2$. We conclude with the following observation. 

\begin{lemma}
  \label{lem:2}
  Consider the hyperplane arrangement defined by the
  equations~\eqref{eq:2} as well as by  $x_i = 0$ and $x_i = u_{i}$ for $1≤i≤n$. 
  Let $K⊆ ℝ^n$ a cell of this hyperplane arrangement and $v ∈
  K$. If $K'$ is the result of scaling $K$ by a factor of $2$ from
  $v$, i.e. 
\begin{displaymath}
     K' = \{ v + 2(w-v) \mid w ∈ K\},
   \end{displaymath}
   then $K'$ satisfies the inequalities $-1/2≤ x_i ≤ u_i + 1/2$ for all $1≤i≤n$.
 \end{lemma}
We use this observation to prove Theorem~\ref{thr:1}:
 \begin{proof}[Proof of Theorem~\ref{thr:1}] 
   The task of~\eqref{eq:1} is to find an integer point in the affine
   subspace defined by the system of equations $Ax = b$ that satisfies
   the bound constraints $0≤ x_i ≤ u_i$. We first partition the
   feasible region with the hyperplanes~\eqref{eq:2} as well as
   $x_i = 0$ and $x_i = u_i$ for each $i$. We then apply the
   approximate integer programming algorithm with approximation factor
   $2$ on each convex set $P_K = \{x \in \setR^n \mid Ax = b\}∩ K$ where $K$ ranges
   over all cells of the arrangement (see Figure~\ref{fig:cells}). In $2^{O(n)}$ time, the
   algorithm either finds an integer point in the convex set $C_K$
   that results from $P_K$ by scaling it with a factor of $2$ from its
   center of gravity, or it asserts that $P_K$ does not contain an
   integer point. Clearly, $C_K ⊆ \{x \in \setR^n \mid  Ax = b\}$ and if the
   algorithm returns an integer point $x^*$, then, by
   Lemma~\ref{lem:2}, this integer point also satisfies the bounds
   $0≤x_i ≤ u_i$.  The running time of the algorithm is equal to the number of cells times $2^{O(n)}$ which is  $2^{O(n)} \prod_{i=1}^n \log_2 (u_i+1)$. 
 \end{proof}
\begin{figure}
 \begin{center}
  \includegraphics[width=12cm]{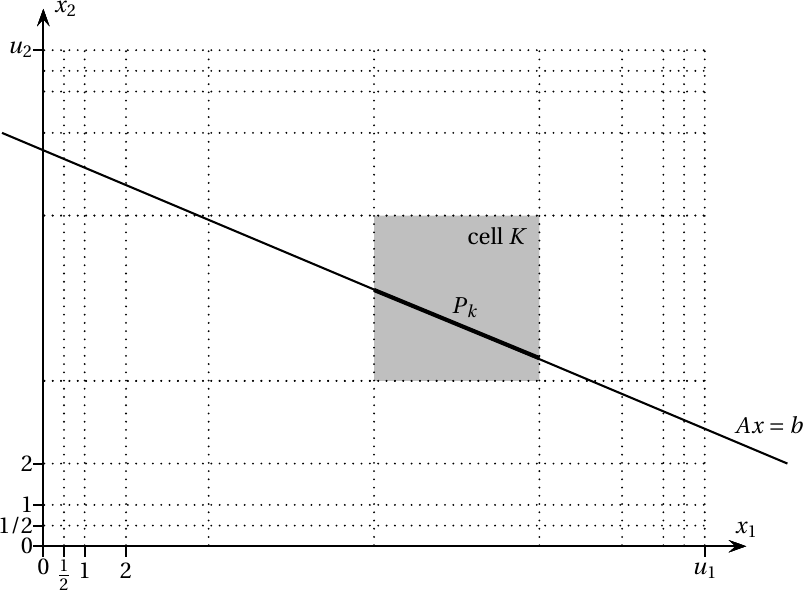}
\end{center}
  \caption{Visualization of the proof of Theorem~\ref{thr:1}.\label{fig:cells}}
\end{figure}

\subsection*{IPs in inequality form}

We can also use Theorem~\ref{thr:1} to solve integer linear programs in \emph{inequality form}. Here
the efficiency is strongly dependent on the number of inequalities.
\begin{theorem} \label{thm:IPinIneqForm}
  Let $A \in \setQ^{m \times n}$, $b \in \setQ^{m}$, $c \in \setQ^n$ and $u \in \setN_+^n$.
  Then the integer linear program
  $$  \max\big\{ \left<c,x\right> \mid Ax \leq b, \; 0 \leq x \leq u, \; x \in \setZ^n \big\} $$
  can be solved in time $n^{O(m)} \cdot (2\log(1+\Delta))^{O(n+m)}$ where $\Delta := \max\{ u_i \mid i=1,\ldots,n\}$.
\end{theorem}
\begin{proof}
  Via binary search it suffices to solve the feasibility problem
  \begin{equation} \label{eq:inequalityIP2}
    \left<c,x\right> \geq \gamma, \; Ax \leq b, \; 0 \leq x \leq u, \; x \in \setZ^n 
  \end{equation}
  in the same claimed running time. We apply the result of Frank and Tardos (Theorem~\ref{thm:FrankTardos}) and
  replace $c,\gamma,A,b$ by integer-valued objects of bounded $\| \cdot \|_{\infty}$-norm so that the feasible region of \eqref{eq:inequalityIP2}
  remains the same. Hence we may indeed assume that $c \in \setZ^n$, $\gamma \in \setZ$, $A \in \setZ^{m \times n}$ and $b \in \setZ^m$ with $\|c\|_{\infty},|\gamma|,\|A\|_{\infty}, \|b\|_{\infty} \leq 2^{O(n^3)} \cdot \Delta^{O(n^2)}$.
  Any feasible solution $x$ to \eqref{eq:inequalityIP2} has a slack bounded by $\gamma-\left<c,x\right> \leq |\gamma| + \|c\|_{\infty} \cdot n \cdot \Delta \leq N$ where we may choose $N := 2^{O(n^3)} \Delta^{O(n^2)}$. Similarly $b_i-\left<A_i,x\right> \leq N$ for all $i \in [n]$.
  We can then introduce slack variables $y \in \setZ_{\geq 0}$ and $z \in \setZ_{\geq 0}^m$ and consider the system
  \begin{equation} \label{eq:equalityIP2}
    \begin{array}{lll}
    \left<c,x\right> + y = \gamma, &  Ax + z = b,  &  \\
    0 \leq x \leq u, & 0\leq y \leq N, & 0 \leq z_j \leq N \; \forall j \in [m],  \\
      (x,y,z) \in \setZ^{n+1+m} & &
                                   \end{array}
  \end{equation}
  in equality form which is feasible if and only if \eqref{eq:inequalityIP2} is feasible.   
  Then Theorem~\ref{thr:1} shows that such an integer linear program can be solved in time
  \[
    2^{O(n+m)} \cdot \Big(\prod_{i=1}^n \ln(1+u_i) \Big) \cdot (\ln(1+N))^{m+1}  \leq n^{O(m)} \cdot (2\log(1+\Delta))^{O(n+m)}.
  \]
\end{proof}

\subsection*{Subset sum and knapsack}
\label{sec:subset-sum-knapsack}

The \emph{subset-sum problem (with multiplicities)} is an integer program of the form~\eqref{eq:1} with one linear constraint.  
 Polak and Rohwedder~\cite{polak2021knapsack} have shown that subset-sum with multiplicities --- that means $\sum_{i=1}^n x_iz_i=t, 0 \leq x_i \leq u_i \; \forall i \in [n], x \in \setZ^n$ --- can be solved in time $O(n +z_{\max}^{5/3})$ times a polylogarithmic factor where $z_{\max} := \max_{i=1,\ldots,n} z_i$. The algorithm of Frank and Tardos~\cite{frank1987application} (Theorem~\ref{thm:FrankTardos}) finds an equivalent instance in which $z_{\max}$ is bounded by $2^{O(n^3)} u_{\max}^{O(n^2)}$. All-together, if each multiplicity is bounded by a polynomial $p(n)$, then the state-of-the-art\footnote{At least at the time of the first submission of this paper.} for subset-sum with multiplicities is  straightforward enumeration resulting in a running time $n^{O(n)}$ which is the current best running time for integer programming. We can significantly improve the running time in this regime. This is a direct consequence of Theorem~\ref{thm:IPinIneqForm}. 

\begin{corollary}
  \label{co:1}
  The subset sum problem with multiplicities of the form $\sum_{i=1}^n x_iz_i=t, 0 \leq x \leq u, x \in \setZ^n$  can be solved in time $2^{O(n)} \cdot (\log(1+\|u\|_{\infty}))^n$. In particular if
   each multiplicity is bounded by a polynomial $p(n)$, then it can be solved in time $(\log n )^{O(n)}$. 
\end{corollary}

\noindent 
\emph{Knapsack} with multiplicities  is the following integer programming problem
\begin{equation}
  \label{eq:3} 
  \max\big\{\left<c, x\right> \mid x ∈ ℤ_{≥0}^n, \, \left<a,x\right> ≤ β, 0 ≤x≤u \big\},
\end{equation}
where $c, a,u ∈ ℤ_{≥0}^n$ are integer vectors.  Again, via the preprocessing algorithm of Frank and Tardos~\cite{frank1987application} (Theorem~\ref{thm:FrankTardos}) one can assume that $\|c\|_∞$ as well as $\|a\|_∞$ are bounded by $2^{O(n^3)} u_{\max}^{O(n^2)}$.  If each $u_i$ is bounded by a polynomial in the dimension, then the state-of-the-art for this problem is again straightforward enumeration which leads to a running time of $n^{O(n)}$. Also in this regime, we can significantly improve the running time
which is an immediate consequence of Theorem~\ref{thm:IPinIneqForm}.
\begin{corollary}
  \label{co:2}
  A knapsack problem~\eqref{eq:3} can be solved in time $2^{O(n)} \cdot (\log (1+\|u\|_{\infty}))^n$.
  In particular if $\|u\|_∞$ is bounded by a polynomial $p(n)$ in the dimension, it can be solved in time $(\log n)^{O(n)}$.  
\end{corollary}

\bibliographystyle{plain}
{\small 
\bibliography{newAlgosForIP,mybib,books}
}
\end{document}